\documentclass[11pt,a4paper]{article}

\usepackage{amssymb,amsmath,amsfonts,amsthm,mathrsfs,mathtools}
\usepackage{bbm,bm}
\usepackage[marginal]{footmisc}
\usepackage{CJK}
\usepackage{mathrsfs}
\usepackage{graphicx}
\usepackage{float}
\usepackage{xcolor}
\usepackage{url}
\usepackage{relsize}
\usepackage{appendix}
\usepackage{epstopdf}
\usepackage{enumitem}

\usepackage[colorlinks,linkcolor=blue,anchorcolor=blue,citecolor=purple,urlcolor=green]{hyperref}

\title{\bf  Asymptotic Formulas for Negative Sobolev Norms and Applications
}
\author{Huaiqian Li
  \vspace{2mm}
\\
{\footnotesize Center for Applied Mathematics and KL-AAGDM, Tianjin University, Tianjin 300072, China}\\
\footnotesize{\texttt{huaiqian.li@tju.edu.cn}}
}

\date{}

\def\R{\mathbb{R}}
\def\E{\mathbb{E}}

\def\d{\textup{d}}

\def\L{\mathrm{L}}
\def\W{\mathrm{W}}

\def\<{\langle}
\def\>{\rangle}
\def\Proof.{\noindent{\bf Proof. }}

\newtheorem{theorem}{Theorem}[section]
\newtheorem{lemma}[theorem]{Lemma}
\newtheorem{corollary}[theorem]{Corollary}

\newtheorem{example}[theorem]{Example}

\theoremstyle{definition}\newtheorem{remark}[theorem]{Remark}

\begin{document}
\allowdisplaybreaks
\maketitle
\makeatletter 
\renewcommand\theequation{\thesection.\arabic{equation}}
\@addtoreset{equation}{section}
\makeatother 

\begin{abstract}
 This paper establishes an asymptotic formula for negative Sobolev norms as the fractional order tends to zero. In the Euclidean setting, under a mild boundedness condition, the product of the order and the norm raised to the power $p$ converges to a dimension-independent constant multiple of the corresponding $L^p$ norm. The proof relies on heat-kernel regularization, weak compactness, and an Abelian--Tauberian argument. The result is further extended to a general measure space equipped with a family of bounded and continuous operators that covering nonlinear and non-semigroup settings. We also present several applications in analysis and probability, including an absolute-continuity criterion for measures,  a random-distribution regularity result, a construction of square-integrable local times for fractional Brownian motion, and limiting formulas for truncated maximal operators and martingales.
\end{abstract}

\section{Introduction}\label{sec-intro}\hskip\parindent
Let $\mathbb{N}$ be the set of positive integers, and fix $n\in\mathbb{N}$. Consider the $n$-dimensional Euclidean space $\R^n$ endowed with the standard inner product $\langle\cdot,\cdot\rangle$ and the induced norm $\|\cdot\|=\langle\cdot,\cdot\rangle^{1/2}$, where $n\in \mathbb{N}$. For $p\in[1,\infty]$, let $\L^p(\R^n)$ be the standard Lebesgue space over $\R^n$, equipped with the norm $\|\cdot\|_{\L^p}$.

Let $\alpha\in(0,1)$ and $p\in[1,\infty)$. Define the standard fractional Sobolev space as
$$\W^{\alpha,p}(\R^n):=\{f\in \L^p(\R^n):\ [f]_{\W^{\alpha,p}}<\infty\},$$
where $[\cdot]_{\W^{\alpha,p}}$ denotes the Gagliardo--Slobodeckij seminorm given by
$$[f]_{\W^{\alpha,p}}=\bigg(\int_{\R^n}\int_{\R^n}\frac{|f(x)-f(y)|^p}{\|x-y\|^{n+\alpha p}}\,\d x\d y\bigg)^{\frac{1}{p}}.$$
For a comprehensive study of fractional Sobolev spaces, we refer the reader to \cite{Leoni2023,DNPV2012}, among others.  It is well known that for every $p\in[1,\infty)$, the family of Gagliardo--Slobodeckij seminorms  $\{[\cdot]_{\W^{\alpha,p}}:\ \alpha\in (0,1)\}$ lacks continuity at the endpoints $\alpha=0$ and $\alpha=1$. Here, following the usual convention, we adopt the notation $[f]_{\W^{0,p}}=\|f\|_{\L^p}$ and $[f]_{\W^{1,p}}=\|\nabla f\|_{\L^p}$, where the gradient $\nabla f$ is understood in the sense of distributions.

In their seminal paper \cite{BBM1}, Bourgain, Brezis, and Mironescu established a fundamental connection between the fractional and classical Sobolev norms over bounded smooth domains. Specifically, a version of their results, formulated on $\R^n$ (see, e.g., \cite{Mohanta24,BSY23,Brezis2002} for proofs), states that for every $p\in[1,\infty)$ and every $f\in \W^{1,p}(\R^n)$, one has
\begin{eqnarray}\label{BBM}
\lim_{\alpha\rightarrow1^-}(1-\alpha)^{\frac{1}{p}}[f]_{\W^{\alpha,p}}=\mathfrak{a}_{p,n}\|\nabla f\|_{\L^p},
\end{eqnarray}
where the constant is given by $\mathfrak{a}_{p,n}=\big(\frac{2}{p}\pi^{\frac{n-1}{2}} \Gamma(\frac{p+1}{2})/\Gamma(\frac{p+n}{2})\big)^{1/p}$, and $\Gamma$ denotes the Gamma function. Subsequently, in a remarkable work \cite{MS2002}, Maz'ya and Shaposhnikova characterized the limiting behavior as $\alpha \to 0^+$. They proved that for every $f\in \cup_{0<\alpha<1}\W^{\alpha,p}(\R^n)$,
\begin{eqnarray}\label{MS}
\lim_{\alpha\rightarrow0^+}\alpha^{\frac{1}{p}}[f]_{\W^{\alpha,p}}=\mathfrak{b}_{p,n}\|f\|_{\L^p},
\end{eqnarray}
where $\mathfrak{b}_{p,n}=\big(\frac{4\pi^{n/2}}{p\Gamma(n/2)}\big)^{1/p}$. In the literature, the limiting formulas in equations \eqref{BBM} and \eqref{MS} are commonly referred to as the BBM formula and the MS formula, respectively.

These groundbreaking results have inspired various extensions across a wide array of settings, generating an extensive body of literature that we are unable to survey exhaustively here. To name just a few contributions, we mention \cite{Oleinik2025,HLYY25,GT2024,DGPYYZ24,Han24,Mohanta24,DLTYY2024,BSY23,ACPS20b,Ludwig2014,Milman05,Davila02,BBM2} for developments concerning the BBM formula, and \cite{NTYYZ25+,LiWu2025+,HPXZ25,DLTYY2024,GT2024,BGT2022,ACPS20a,Ludwig2014,Milman05} for advances related to the MS formula.

Turning now to the question of interest in the present note, we focus on the negative-order regime. For $\alpha>0$ and $p\in(1,\infty)$, the negative Sobolev space $\W^{-\alpha,p}(\R^n)$ is defined as the dual space of $\W^{\alpha,q}(\R^n)$, where $q$ is the H\"older conjugate exponent of $p$ (see, e.g., \cite[Section 6.5]{Leoni2023}). These spaces of negative differentiability have been the subject of intensive investigation, and play a crucial role in function analysis and theory of (stochanstic) partial differential equations; see, e.g., \cite{FL2023,Mitrea2018,Brezis2011,Temam2001}. Given the significance of the limiting formulas for positive-order spaces, it is natural to ask whether an analogous continuity property holds for the dual norms. Specifically, for $p\in(1,\infty)$, we pose the following question:
\begin{itemize}[leftmargin=2.2cm]
\item[(\textbf{Q})]  Does the family of norms $\{\|\cdot\|_{\W^{-\alpha,p}}\}_{\alpha\in(0,1)}$ depend continuously on $\alpha$ at the endpoint $\alpha=0$?
\end{itemize}

The remainder of the paper is structured as follows. Section \ref{sec-nms} is devoted to question (\textbf{Q}), with the main result stated and proved in Theorem \ref{main-MS}. Building on this foundation, Section \ref{sec-extension} extends the investigation to a broader framework and establishes a corresponding limiting formula in Theorem \ref{main-gen-MS}. Finally, Section \ref{sec-applications} presents several applications of the two limiting formulas.

\section{Limiting formulas for negative Sobolev norms}\label{sec-nms}\hskip\parindent
In order to address the question (\textbf{Q}), we recall an equivalent definition of the negative Sobolev norm based on the heat kernel (also known as the Gaussian--Weierstrass function); see  \cite[Proposition D.1]{AKM2019}. For $t>0$, let $\Phi_t$ denote the heat kernel given by
\begin{align}\label{GWF}
\Phi_t(x)=\frac{1}{(\sqrt{4\pi t})^{n}}\exp\bigg(-\frac{\|x\|^2}{4 t}\bigg),\quad x\in\R^n.
\end{align}
For $p\in(1,\infty)$ and $\alpha>0$, the negative Sobolev norm on $\W^{-\alpha,p}(\R^n)$ is defined for $f\in \W^{-\alpha,p}(\R^n)$ by
\begin{align}\label{neg-sob-norm}
\|f\|_{\W^{-\alpha,p}}=
\begin{dcases}
\left(
  \int_0^1 \left(t^{\frac{\alpha}{2}}\|f\ast\Phi_t\|_{\L^p}\right)^p \frac{\mathrm{d}t}{t}
\right)^{\frac{1}{p}},& \alpha\notin\mathbb{N},\\[1.2ex]
\left\|\left(\int_0^1 \left(t^{\frac{\alpha}{2}} |f\ast \Phi_t|\right)^2 \frac{\mathrm{d}t}{t}
  \right)^{\frac{1}{2}}\right\|_{{\rm L}^p},& \alpha\in\mathbb{N},
\end{dcases}
\end{align}
where $\ast$ indicates the convolution operation. Here and below, we use the convention $\|f\|_{\W^{-\alpha,p}}=+\infty$ if $f\notin\W^{-\alpha,p}(\R^n)$. For other equivalent definitions of the negative Sobolev norm, see \cite{Grafakos2014b,Triebel92} for instance. Since we shall be concerned exclusively with the limiting case $\alpha\to0^+$, only the first branch of \eqref{neg-sob-norm} (where $\alpha\notin\mathbb{N}$) will be needed in what follows. Throughout the paper, when a distribution $f$ does not belong to $\W^{-\alpha,p}(\R^n)$, we set $\|f\|_{\W^{-\alpha,p}}:=+\infty.$

The main result of this section,  presented in the next theorem, establishes a limiting formula for the negative Sobolev norms $\|\cdot\|_{\W^{-\alpha,p}}$ as $\alpha\rightarrow0^+$, which is notably dimension-free.
\begin{theorem}\label{main-MS}
Let $p\in(1,\infty)$. Suppose that $f\in \cup_{\alpha\in(0,1)}\W^{-\alpha,p}(\R^n)$ and
\begin{eqnarray*}
\liminf_{\alpha\rightarrow0^+}\alpha\|f\|_{\W^{-\alpha,p}}^p<\infty.
\end{eqnarray*}
Then, $f\in\L^p(\R^n)$, the limit $\lim_{\alpha\rightarrow0^+}\alpha\|f\|_{\W^{-\alpha,p}}^p$ exists, and
\begin{eqnarray}\label{main-MS-1}
\lim_{\alpha\rightarrow0^+}\alpha\|f\|_{\W^{-\alpha,p}}^p=\frac{2}{p} \|f\|^p_{\L^p}.
\end{eqnarray}
\end{theorem}

In order to prove Theorem \ref{main-MS}, we need to recall some notions. Let $\mathcal{S}(\R^n)$ be the Schwartz space of rapidly decreasing smooth functions on $\R^n$, and let $\mathcal{S}'(\R^n)$ be its dual space, i.e., the space of tempered distributions. For a sequence $(\chi_k)_{k\in\mathbb{N}}\subseteq \mathcal{S}(\R^n)$ and $\chi\in\mathcal{S}(\R^n)$, we say that $(\chi_k)_{k\in\mathbb{N}}\subseteq \mathcal{S}(\R^n)$ converges to $\chi$ in $\mathcal{S}(\R^n)$ if for all multi-indices $\bm{\alpha},\bm{\beta}\in (\mathbb{N}\cup\{0\})^n$ we have
$$\sup_{x\in\R^n}\left|x^{\bm{\alpha}}[\partial^{\bm{\beta}}(\chi_k- \chi)](x)\right|\to 0,\quad \mbox{as }k\to\infty,$$
where for $x=(x_1,\cdots,x_n)\in\R^n$ and $\bm{\alpha}=(\alpha_1,\cdots,\alpha_n)\in (\mathbb{N}\cup\{0\})^n$, we understand $x^{\bm{\alpha}}=x_1^{\alpha_1}\cdots x_n^{\alpha_n}$ and $\partial^{\bm{\alpha}} \chi=\partial_1^{\alpha_1}\cdots \partial_n^{\alpha_n} \chi$. For a tempered distribution $f\in\mathcal{S}'(\R^n)$ and a test function $\chi\in\mathcal{S}(\R^n)$, the action of $f$ on $\chi$ is denoted by $(f,\chi)$, and the convolution $f\ast \chi$ is defined by
\begin{align*}
(f\ast\chi, \varphi)=(f, \tilde{\chi}\ast \varphi),\quad \varphi\in\mathcal{S}(\R^n),
\end{align*}
where $\tilde{\chi}(x):=\chi(-x)$ for any $x\in\R^n$. A sequence $(f_k)_{k\in\mathbb{N}}\subseteq\mathcal{S}'(\R^n)$ is said to converge to $f\in \mathcal{S}'(\R^n)$ in the sense of distributions (i.e., in the weak$^\ast$-topology of $\mathcal{S}'(\R^n)$) if
\begin{align*}
\lim_{k\to\infty}(f_k,\chi)=(f,\chi),\quad \chi\in\mathcal{S}(\R^n).
\end{align*}
For more details on the tempered distribution, refer to \cite{Mitrea2018,Grafakos2014a} for instance.

We need the following results, which are more or less  standard.
\begin{lemma}\label{diff-lem}
Let $p\in[1,\infty)$. Then the following assertions hold.
\begin{itemize}
\item[(1)] For every $f\in \L^p(\mathbb R^n)$,
$$\|f-f\ast\Phi_t\|_{\L^p}\rightarrow0\quad\mbox{as }t\to0^+.$$

\item[(2)] For every \(f\in\mathcal S'(\mathbb R^n)\),
$$f\ast\Phi_t\longrightarrow f\quad\mbox{in }\mathcal S'(\mathbb R^n)$$
as $t\to0^+$. If, in addition, there exists $\alpha>0$ such that
\begin{align}\label{diff-lem-assumption}
\int_0^1t^{\frac{\alpha p}{2}-1}\|f\ast\Phi_t\|_{\L^p}^p\,\d t<\infty,
\end{align}
 then $f\ast\Phi_t\in \L^p(\mathbb R^n)$ for every $t>0$, and
$$(0,\infty)\ni t\longmapsto f\ast\Phi_t\in \L^p(\mathbb R^n)$$
is continuous.
\end{itemize}
\end{lemma}
\begin{proof}
(i) Since
$$\Phi_t(x)=t^{-n/2}\Phi_1\Big(\frac{x}{\sqrt t}\Big),\quad\int_{\mathbb R^n}\Phi_1(x)\,\d x=1,$$
the family $(\Phi_t)_{t>0}$ is an approximate identity in the sense of \cite[Definition~1.2.15]{Grafakos2014a}. As a result, part (1) follows directly from \cite[Theorem 1.2.19(1)]{Grafakos2014a}.

(ii) We now prove the first assertion in part (2). For every $\chi\in\mathcal S(\mathbb R^n)$, applying \cite[Exercise 2.3.2]{Grafakos2014a} with
$\varepsilon=\sqrt t$, we have
$$\Phi_t\ast\chi\rightarrow\chi\quad\mbox{in }\mathcal S(\mathbb R^n)$$
as $t\to0^+$. Since the heat kernel is symmetric, we obtain
$$\tilde{\Phi}_t(x)=\Phi_t(-x)=\Phi_t(x),\quad x\in\R^n,\,t>0,$$
and hence, as $t\to0^+$,
$$(f\ast\Phi_t,\chi)=( f,\tilde{\Phi}_t\ast\chi)=( f,\Phi_t\ast\chi)\rightarrow (f,\chi),$$
Thus,
$$f\ast\Phi_t\rightarrow f\quad\mbox{in }\mathcal S'(\mathbb R^n),$$
as $t\to0^+$.

(iii) It remains to prove the second assertion in part (2). For $g\in \L^p(\mathbb R^n)$ and $r>0$, Young's inequality and $\|\Phi_r\|_{\L^1}=1$ imply
$$\|g\ast\Phi_r\|_{\L^p}\leq\|g\|_{\L^p}.$$
Together with part~\textup{(1)}, these imply that
$$(0,\infty)\ni r\mapsto g\ast\Phi_r\in \L^p(\mathbb R^n)$$
is continuous.

By the assumption \eqref{diff-lem-assumption}, $f\ast\Phi_s\in \L^p(\mathbb R^n)$ for a.e. $s\in(0,1)$. Fix $t_0>0$, and choose $0<s<\min\left\{t_0/2,1\right\}$ such that $g:=f\ast\Phi_s\in \L^p(\mathbb R^n)$. For every $t>s$, the semigroup property yields
$$f\ast\Phi_t=(f\ast\Phi_s)\ast\Phi_{t-s}=g\ast \Phi_{t-s}.$$
Thus $f\ast\Phi_t\in \L^p(\mathbb R^n)$ for $t>s$, and the map $t\mapsto f\ast\Phi_t$ is continuous near $t_0$ in $\L^p(\R^n)$. Since $t_0>0$ is arbitrary, the desired conclusion follows.
\end{proof}

\begin{remark}\label{remark-diff-lem}
In the remainder of this section, we do not use the second assertion of Lemma \ref{diff-lem}(2); however, see Remark \ref{remark-cor-random} for an application of this assertion.
\end{remark}

With the preceding preparations in place, we are now ready to present the proof of Theorem \ref{main-MS}.
\begin{proof}[Proof of  Theorem \ref{main-MS}]
Let $p\in(1,\infty)$, $\alpha\in(0,1)$ and  $f\in \cup_{\alpha\in(0,1)}\W^{-\alpha,p}(\R^n)$. Set $\beta:= \frac{\alpha p}{2}$, and set $\phi(t):=\|f\ast\Phi_t\|_{\L^p}^p$ for any $t>0$. Then
\begin{align}\label{pf-main-ms-00}
\alpha \|f\|_{\W^{-\alpha,p}}^p&=\frac{2\beta}{p}\int_0^1 t^{\beta-1}\phi(t)\,\d t\cr
&=\frac{2\beta}{p}\int_0^\infty e^{-\beta s}\phi(e^{-s})\,\d s\cr
&=\frac{2\beta}{p}\int_0^\infty e^{-\beta s}\psi(s)\,\d s,
\end{align}
where we used the change of variables $t=e^{-s}$ and then denoted $\psi(s)=\phi(e^{-s})$.

Observe that $t\mapsto\phi(t)$ is non-increasing on $(0,\infty)$. Indeed, for any $0< s<t$, by Young's inequality, we have
$$\phi(t)=\|f\ast\Phi_t\|_{\L^p}^p=\|f\ast\Phi_{s}\ast\Phi_{t-s}\|_{\L^p}^p\leq \|f\ast\Phi_{s}\|_{\L^p}^p=\phi(s),$$
where we also used the semigroup property $\Phi_{t+s}=\Phi_{t}\ast\Phi_{s}$ and the fact that $\|\Phi_s\|_{\L^1}=1$ for any $s,t>0$. Hence, $t\mapsto\psi(t)$ is a non-decreasing function on $(0,\infty)$.

By the assumption $\liminf_{\alpha\rightarrow0^+}\alpha\|f\|_{\W^{-\alpha,p}}^p<\infty$, it follows from \eqref{pf-main-ms-00} that there exist a constant $C>0$ and a sequence $\beta_k\rightarrow0^+$ as $k\rightarrow\infty$ such that
\begin{align*}
\beta_k\int_0^\infty e^{-\beta_k s}\psi(s)\,\d s\leq C.
\end{align*}
We \textbf{claim} that $\sup_{s\in[0,\infty)}\psi(s)<\infty$. If not, then $\psi$ is unbounded as $s\to\infty$. Since $\psi$ is non-decreasing, we have $\lim_{s\rightarrow\infty}\psi(s)=\infty$. For every $M>0$, there exists some number $N>0$ such that $\psi(s)\geq M$ for all $s\geq N$. Then
\begin{align*}
\beta\int_0^\infty e^{-\beta s}\psi(s)\,\d s\geq M\beta\int_{N}^\infty e^{-\beta s}\,\d s=Me^{-N\beta},\quad \beta>0.
\end{align*}
Taking $\beta=\beta_k$ and letting $k\to \infty$ leads to
\begin{align*}
\liminf_{k\rightarrow\infty}\beta_k\int_0^\infty e^{-\beta_k s}\psi(s)\,\d s\geq M,
\end{align*}
which contradicts the assumption since $M$ is arbitrary. Thus, the \textbf{claim} is proved. Since a bounded monotone function has a limit, we denote
\begin{align}\label{pf-main-ms-a}
L=\lim_{s\rightarrow\infty} \psi(s)\in[0,\infty).
\end{align}

We now prove that
\begin{align}\label{pf-main-ms-0}
\lim_{\beta\rightarrow0^+}\beta \int_0^\infty e^{-\beta  s}\psi(s)\,\d s
=\lim_{\beta\rightarrow0^+}\int_0^\infty e^{-u}\psi\left(\frac{u}{\beta}\right)\,\d u
=L.
\end{align}
The first equality follows from the change of variables $u=\beta s$, and the second equality can be derived by applying the dominated convergence theorem due to the boundedness of $\psi$ and \eqref{pf-main-ms-a} (see also Remark \ref{AT}).

Combining \eqref{pf-main-ms-0} with \eqref{pf-main-ms-00}, we have
\begin{align}\label{pf-main-ms-3}
\lim_{\alpha\rightarrow0^+}\alpha \|f\|_{\W^{-\alpha,p}}^p=\frac{2}{p}L.
\end{align}

It remains to show that $f\in \L^p(\R^n)$ and $L=\|f\|_{\L^p}^p$. Since $f\in\W^{-\alpha_0,p}(\R^n)$ for some $\alpha_0\in(0,1)$, it follows that $f$ is a tempered distribution. To see this, by \cite[Theorem 6.74]{Leoni2023} or \cite[Proposition 2.3.23]{Grafakos2014a}, the Schwartz space $\mathcal{S}(\R^n)$ is dense in $\W^{\alpha_0,q}(\R^n)$, where $q=\frac{p}{p-1}$. Consequently, the dual space $\W^{-\alpha_0,p}(\R^n)=\left(\W^{\alpha_0,q}(\R^n)\right)'$ embeds continuously into $\mathcal{S}'(\R^n)$. Hence, every element of $\W^{-\alpha_0,p}(\R^n)$ is a tempered distribution.

Since $f$ is a tempered distribution, by employing Lemma \ref{diff-lem}(2), we see that $f\ast\Phi_t\rightarrow f$ in the sense of distributions as $t\rightarrow0^+$. Moreover, \eqref{pf-main-ms-a} and the monotonicity of $\phi$ show that  the family $\{f\ast\Phi_t:\ t>0\}$ is bounded in $\L^p(\R^n)$. Since $\L^p(\R^n)$ is reflexive whenever $p\in(1,\infty)$, by the Banach--Alaoglu theorem, there exists a sequence $t_k\rightarrow0^+$ as $k\rightarrow\infty$ such that $f\ast\Phi_{t_k}$ converges weakly in $\L^p(\R^n)$ to some function $g\in \L^p(\R^n)$ as $k\rightarrow\infty$.  By uniqueness of distributional limits, we have $f=g\in\L^p(\R^n)$.

Finally, since $f\in \L^p(\R^n)$, it follows from Lemma \ref{diff-lem}(1) that $f\ast\Phi_t\to f$ in $\L^p(\R^n)$ as $t\to0^+$.  Combining this with \eqref{pf-main-ms-a}  and \eqref{pf-main-ms-3}, we obtain \eqref{main-MS-1}.

Therefore, we complete the proof.
\end{proof}

We give some remarks on Theorem \ref{main-MS} and its proof.
\begin{remark}\label{AT}
(1) In the proof of Theorem \ref{main-MS}, \eqref{pf-main-ms-a} and \eqref{pf-main-ms-0} are indeed equivalent, which follows directly from a more general theory, namely the Abelian--Tauberian theorem (see  \cite[Theorem 2, Section 5, Chapter XIII]{Feller2}), which can be adapted into the following simpler form. \emph{Let $\psi : [0, \infty) \to [0, \infty)$ be a non-decreasing function such that for every $\beta > 0$ the integral
$$\int_0^\infty e^{-\beta s} \psi(s) \, \d s$$
converges. Then the following two limits are equivalent:
\begin{enumerate}
\item[(i)] $\displaystyle \lim_{\beta \to 0^+} \beta \int_0^\infty e^{-\beta s} \psi(s) \, \d s = A$ (exists and is finite);
\item[(ii)] $\displaystyle \lim_{s \to \infty} \psi(s) = A$.
\end{enumerate}}
\noindent For a detailed treatment of Abelian and Tauberian theorems, see, e.g., \cite[Chapter II.7]{Tenenbaum15}.

(2) By using the Fourier transform, we can easily derive the limiting formula in Theorem \ref{main-MS} in the particular Hilbert case ($p=2$), i.e.,
\begin{eqnarray*}
\lim_{\alpha\rightarrow0^+}\alpha\|f\|_{\W^{-\alpha,2}}^2=\|f\|^2_{\L^2},\quad f\in\L^2(\R^n).
\end{eqnarray*}
\begin{proof}
Let $f\in\L^2(\R^n)$, and denote its Fourier transform by $\hat{f}(\xi)=\int_{\R^n} f(x)e^{-i\langle x,\xi\rangle}\,\d x$ for any $\xi\in\R^n$. Recall that $\hat{\Phi}_t(\xi)=e^{-t\|\xi\|^2}$ and $\widehat{f\ast\Phi_t}(\xi)=\hat{f}(\xi)\hat{\Phi}_t(\xi)$ for any $t>0$ and any $\xi\in\R^n$. Using Plancherel's identity and the Fubini--Tonelli theorem, we compute
\begin{eqnarray*}
\alpha\|f\|_{\W^{-\alpha,2}}^2&=&\alpha\int_0^1 t^{\alpha-1}\|\hat{f}\hat{\Phi}_t\|_{\L^2}^2\,\d t\cr
&=& \int_{\R^n} |\hat{f}(\xi)|^2\bigg(\int_0^1 \alpha  t^{\alpha-1} e^{-2t\|\xi\|^2}\,\d t\bigg)\,\d\xi\cr
&=&\int_{\R^n} |\hat{f}(\xi)|^2{\rm J}(\alpha,\xi)\,\d\xi,\quad \alpha>0,
\end{eqnarray*}
where we denoted ${\rm J}(\alpha,\xi)= \alpha\int_0^1 t^{\alpha-1} e^{-2t\|\xi\|^2}\,\d t$.  For every $\alpha>0$ and every $\xi\in\R^n$, we have the uniform bound $|{\rm J}(\alpha,\xi)|\leq \alpha\int_0^1 t^{\alpha-1}\,\d t=1$. Now fix $\xi\in\R^n$. To find the pointwise limit of ${\rm J}(\alpha,\xi)$ as $\alpha\to0^+$, by the elementary inequality $|e^{-a}-1|\leq a$ for any $a\geq0$, we estimate
\begin{eqnarray*}
|{\rm J}(\alpha,\xi)-1|&=&\left|\alpha\int_0^1 t^{\alpha-1} (e^{-2t\|\xi\|^2}-1)\,\d t\right|\cr
&\leq& 2\|\xi\|^2\alpha\int_0^1 t^{\alpha} \,\d t=2\|\xi\|^2\frac{\alpha}{\alpha+1},\quad \alpha>0.
\end{eqnarray*}
which clearly implies that for every $\xi\in\R^n$, $\lim_{\alpha\to0^+}{\rm J}(\alpha,\xi)=1$. Combining these with the dominated convergence theorem, we conclude that
\begin{eqnarray*}
\lim_{\alpha\to0^+}\alpha\|f\|_{\W^{-\alpha,2}}^2=\int_{\R^n} |\hat{f}(\xi)|^2\,\d\xi=\|f\|_{\L^2}^2,
\end{eqnarray*}
where the last equality is again Plancherel's identity.
\end{proof}
\end{remark}

\section{An abstract limiting principle}\label{sec-extension}\hskip\parindent
The previous result relies heavily on the standard heat kernel and the specific structure of $\R^n$. In this section, by adopting a new approach, we generalize the asymptotic formula in Theorem \ref{main-MS} to the setting of abstract measure spaces equipped with a family of general operators.

Let $(E,\mathcal{E})$ be a measurable space endowed with a $\sigma$-finite measure $\mu:\mathcal{E}\rightarrow[0,\infty]$.  For $p\in[1,\infty)$, let $\L^p(E,\mu)$ be the Lebesgue space over $(E,\mathcal{E},\mu)$ endowed with the norm $\|\cdot\|_{\L^p(\mu)}$ given by
$$\|f\|_{\L^p(\mu)}=\bigg(\int_E |f(x)|^p\,\mu(\d x)\bigg)^{1/p},\quad p\in[1,\infty).$$

Let $\mathcal{B}(0,\infty)$ be the Borel $\sigma$-algebra on $(0,\infty)$. Fix $p\in[1,\infty)$. Let $(S_t)_{t>0}$ be a family of operators on $\L^p(E,\mu)$ satisfying the following assumptions:
\begin{itemize}[leftmargin=2.2cm]
\item[$(\mathfrak{p}.1)$](\emph{Measurability}) for every $f\in \L^p(E,\mu)$, the functions $S_tf$ admit representatives for which $(0,\infty)\times E\ni (t,x)\mapsto S_tf(x)$ is jointly measurable with respect to the product $\sigma$-algebra $\mathcal{B}(0,\infty)\times\mathcal{E}$;

\item[$(\mathfrak{p}.2)$](\emph{Boundedness}) there exists a constant $C>0$ such that
$$\|S_tf\|_{\L^p(\mu)}\leq C\|f\|_{\L^p(\mu)},\quad t\in(0,1],\,f\in \L^p(E,\mu);$$

\item[$(\mathfrak{p}.3)$](\emph{Continuity}) $\lim_{t\to0^+}\|S_tf\|_{\L^p(\mu)}=\|f\|_{\L^p(\mu)}$  for every $f\in \L^p(E,\mu)$.
\end{itemize}

\begin{remark}
A canonical source of examples is give by strongly continuous contraction semigroups on $\L^p(E,\mu)$, including heat semigroups and Markov semigroups associated with Dirichlet forms, as commonly encountered in functional analysis and probability theory; see, e.g., \cite{BGL2014,FOT2011,MR1992,Yoshida}. In contrast with the proof of Theorem \ref{main-MS}, however, the abstract argument below does not require the operators $S_t$ to be linear, nor does it rely on the contraction property (which requires $\|S_tf\|_{\L^p(\mu)} \leq \|f\|_{\L^p(\mu)}$ for any $t>0$ and any $f\in \L^p(E,\mu)$),  the semigroup property (i.e., $S_0=I$ and $S_{t_1}S_{t_2}=S_{t_1+t_2}$ for any $t_1,t_2\geq0$, where $I$ indicates the identity operator), or the strong continuity (which demands $\lim_{t\to0^+}\|S_tf-f\|_{\L^p(\mu)}=0$ for any $f\in \L^p(E,\mu)$).
\end{remark}

For $\alpha>0$ and $p\in[1,\infty)$, we introduce the following negative-order functional
\begin{eqnarray}\label{general-norm}
\|f\|_{\mathcal{W}^{-\alpha,p}(\mu)}=\bigg(\int_0^1 t^{\frac{\alpha p}{2}-1}\| S_tf\|_{\L^p(\mu)}^p\,\d t\bigg)^{\frac{1}{p}},\quad f\in\L^p(E,\mu).
\end{eqnarray}
A simple estimate shows that for every $f\in \L^p(E,\mu)$, property $(\mathfrak{p}.2)$ yields
\begin{eqnarray*}
\|f\|_{\mathcal{W}^{-\alpha,p}(\mu)}^p
&=&\int_0^1 t^{\frac{\alpha p}{2}-1}\| S_tf\|_{\L^p(\mu)}^p\,\d t\\
&\leq& C\|f\|_{\L^p(\mu)}^p\int_0^1 t^{\frac{\alpha p}{2}-1}\,\d t=\frac{2}{p\alpha}C \|f\|_{\L^p(\mu)}^p<\infty,
\end{eqnarray*}
for some constant $C>0$ (independent of $t$ and $f$).

If, in addition, the family $(S_t)_{t>0}$ possesses the following properties: for every  $t>0$,
$$\|S_t(\lambda f)\|_{\L^p(\mu)}= |\lambda| \|S_t f\|_{\L^p(\mu)},\quad\lambda\in\R,\,f\in \L^p(E,\mu),$$
and
$$\|S_t(f+g)\|_{\L^p(\mu)}\leq \|S_t f\|_{\L^p(\mu)}+\|S_t g\|_{\L^p(\mu)},\quad f,g\in \L^p(E,\mu),$$
 then one can directly verify that $\|\cdot\|_{\mathcal{W}^{-\alpha,p}(\mu)}$ defines a norm on $\L^p(E,\mu)$. In particular, these conditions are automatically satisfied if each $S_t$ is a linear operator.

In particular, a concrete and important example arises when  $E=\R^n$, $\mathcal{E}$ is the Borel $\sigma$-algebra on $\R^n$, $\mu$ is the $n$-dimensional Lebesgue measure, and $(S_t)_{t>0}$ is the standard heat semigroup on $\R^n$, i.e., for $f\in \L^p(\R^n)$, define
$$S_t f=f\ast\Phi_t,\quad t>0,$$
where $\Phi_t$ is the heat kernel given by \eqref{GWF}. Then, for $\alpha\in(0,1)$ and $p\in(1,\infty)$, we have the identification
$$\|f\|_{\mathcal{W}^{-\alpha,p}(\mu)}=\|f\|_{\W^{-\alpha,p}},\quad f\in\L^p(\R^n),$$
where the negative Sobolev norm in the right-hand side is defined by \eqref{neg-sob-norm} when $\alpha\notin\mathbb{N}$. This observation provides the natural motivation for introducing the general definition \eqref{general-norm}.

The main result is the following partial extension of Theorem \ref{main-MS}, which generalizes the limiting formula \eqref{main-MS-1} in two ways: an abstract formulation, and the $p = 1$ endpoint for $f \in \L^p(E,\mu)$.
\begin{theorem}\label{main-gen-MS}
Let $p\in[1,\infty)$. For every $f\in\L^p(E,\mu)$,
\begin{equation*}
\lim_{\alpha\rightarrow0^+}\alpha\|f\|_{\mathcal{W}^{-\alpha,p}(\mu)}^p=\frac{2}{p}\|f\|_{\L^p(\mu)}^p.
\end{equation*}
\end{theorem}

Although the operators $S_t$ are allowed to be nonlinear and need not form a semigroup, the proof of Theorem \ref{main-gen-MS} only uses the measurability, the boundedness and the endpoint continuity of the scalar function $t\mapsto \|S_tf\|_{\L^p(\mu)}$. We include the short argument for completeness.
\begin{proof}[Proof of Theorem \ref{main-gen-MS}]
Let $p\in[1,\infty)$ and $f\in \L^p(E,\mu)$. Set
$$\beta:=\frac{\alpha p}{2},\quad\quad g(t):=\|S_tf\|_{\L^p(\mu)}^p,\quad t\in(0,1].$$
By $(\mathfrak{p}.1)$ and the Fubini--Tonelli theorem, $g$ is measurable. By $(\mathfrak{p}.2)$,
$$0\leq g(t)\leq C^p\|f\|_{\L^p(\mu)}^p,\quad t\in(0,1],$$
where $C>0$ is a constant, and by $(\mathfrak{p}.3)$,
$$\lim_{t\to0^+}g(t)=\|f\|_{\L^p(\mu)}^p.$$
Using the changes of variables $t=e^{-s}$ and then $u=\beta s$, we obtain
\begin{align*}
\alpha\|f\|_{\mathcal{W}^{-\alpha,p}(\mu)}^p
&=\alpha\int_0^1 t^{\beta-1}g(t)\,\d t\\
&=\alpha\int_0^\infty e^{-\beta s}g(e^{-s})\,\d s\\
&=\frac{2}{p}\int_0^\infty e^{-u}g(e^{-u/\beta})\,\d u.
\end{align*}
For every $u>0$, $e^{-u/\beta}\to0$ as $\beta\to0^+$. Hence,
$$g(e^{-u/\beta})\to\|f\|_{\L^p(\mu)}^p,\quad\mbox{as }\beta\to0^+.$$
Moreover, the integrand $e^{-u}g(e^{-u/\beta})$ is dominated by $C^p e^{-u}\|f\|_{\L^p(\mu)}^p$. The dominated convergence theorem gives
\begin{align*}
\lim_{\alpha\to0^+}\alpha\|f\|_{\mathcal{W}^{-\alpha,p}(\mu)}^p=\frac{2}{p}\int_0^\infty e^{-u}\|f\|_{\L^p(\mu)}^p\,\d u
=\frac{2}{p}\|f\|_{\L^p(\mu)}^p.
\end{align*}
This proves the theorem.
\end{proof}

\section{Applications}\label{sec-applications}\hskip\parindent
In this section we provide several consequences and examples of Theorems \ref{main-MS} and \ref{main-gen-MS}. The first group of applications uses the part of Theorem \ref{main-MS} that a distribution whose negative Sobolev norms have the critical growth rate $\alpha^{-1/p}$ must in fact be an $L^p$ function. The second group illustrates the flexibility of the abstract formula in Theorem \ref{main-gen-MS} by applying it to nonlinear maximal operators and martingales.

\subsection{Regularity and absolute continuity consequences}
\begin{corollary}[A criterion for absolute continuity]\label{cor-abs-cont}
Let $p\in(1,\infty)$, and let $\nu$ be a finite signed Radon measure on $\R^n$, regarded as an element of $\mathcal{S}'(\R^n)$. Suppose that $\nu\in \cup_{\alpha\in(0,1)}\W^{-\alpha,p}(\R^n),$  and
\begin{align}\label{cor-liminf-cond}
\liminf_{\alpha\to0^+}\alpha\|\nu\|_{\W^{-\alpha,p}}^p<\infty.
\end{align}
Then $\nu$ is absolutely continuous with respect to the $n$-dimensional Lebesgue measure. More precisely, there exists $f_\nu\in\L^p(\R^n)$ such that
$\nu=f_\nu \d x,$ and
\begin{align*}
\lim_{\alpha\to0^+}\alpha\|\nu\|_{\W^{-\alpha,p}}^p
=\frac{2}{p}\|f_\nu\|_{\L^p}^p.
\end{align*}
\end{corollary}
\begin{proof}
From the assumptions, we may apply Theorem \ref{main-MS} to the distribution $\nu$. Hence $\nu$ coincides, as a distribution, with some function $f_\nu\in\L^p(\R^n)$, and the limiting identity follows from \eqref{main-MS-1}. Since $\nu$ is originally a finite signed measure, the equality of distributions implies the equality of measures $\nu=f_\nu \d x$.
\end{proof}

\begin{example}[Dirac measures]\label{eg-dirac}
Corollary \ref{cor-abs-cont} rules out singular measures. This can be seen explicitly for the Dirac measure $\delta_0$ at the point $0\in\R^n$. For $t>0$, since $\delta_0\ast\Phi_t=\Phi_t$, we have
\begin{align*}
\|\delta_0\ast\Phi_t\|_{\L^p}^p=p^{-n/2}(4\pi t)^{-n(p-1)/2}.
\end{align*}
Consequently,
\begin{align*}
\|\delta_0\|_{\W^{-\alpha,p}}^p=p^{-n/2}(4\pi)^{-n(p-1)/2} \int_0^1 t^{\frac{\alpha p-n(p-1)}{2}-1}\,\d t.
\end{align*}
This integral is finite if and only if
$$\alpha>\frac{n(p-1)}{p}.$$
In particular, no neighborhood of $\alpha=0$ can satisfy the condition \eqref{cor-liminf-cond}.
\end{example}

The next result may be useful in the study of sample paths with limited spatial regularity, for example in establishing density properties of empirical or occupation measures (see Example \ref{eg-fbm}) and in the investigation of stochastic partial differential equations, particularly in fluid mechanics (see, e.g., \cite{FL2023}). For a comprehensive study of Lebesgue--Bochner space, refer to \cite{HvNVW2016}.
\begin{corollary}[A random-distribution criterion]\label{cor-random}
Let $p\in(1,\infty)$, and let $X$ be an $\mathcal{S}'(\R^n)$-valued random variable defined on a probability space $(\Omega,\mathcal{F},\mathbb{P})$. Suppose that there exists a sequence $(\alpha_j)_{j\in\mathbb{N}}\subset(0,1)$ such that $\alpha_j\downarrow0$ and $X\in\W^{-\alpha_j,p}(\R^n)$ almost surely for every $j\in\mathbb{N}$, and
\begin{align}\label{cor-random-cond}
\sup_{j\in\mathbb{N}}\E\left[\alpha_j\|X\|_{\W^{-\alpha_j,p}}^p\right]<\infty.
\end{align}
Assume moreover that, for every $t>0$, the convolution $X\ast\Phi_t$ admits a strongly measurable $\L^p(\R^n)$-valued version. Then there exists an $\L^p(\R^n)$-valued random variable $Y$ such that $X=Y$ in $\mathcal{S}'(\R^n)$ almost surely.
Moreover,
\begin{align}\label{cor-random-est}
\E\|Y\|_{\L^p}^p\leq \frac{p}{2}\liminf_{j\to\infty}\E\left[\alpha_j\|X\|_{\W^{-\alpha_j,p}}^p\right].
\end{align}
\end{corollary}
\begin{proof}
By Fatou's lemma and \eqref{cor-random-cond},
\begin{align*}
\E\left[\liminf_{j\to\infty}\alpha_j\|X\|_{\W^{-\alpha_j,p}}^p\right]
\leq\liminf_{j\to\infty}\E\left[\alpha_j\|X\|_{\W^{-\alpha_j,p}}^p\right]
<\infty.
\end{align*}
Hence, for almost every $\omega\in\Omega$,
\begin{align}\label{cor-random-path}
\liminf_{j\to\infty}\alpha_j\|X(\omega)\|_{\W^{-\alpha_j,p}}^p<\infty.
\end{align}
For such an $\omega$, Theorem \ref{main-MS} applies to the distribution $X(\omega)$ and gives an element $Y(\omega)\in\L^p(\R^n)$ such that
$$X(\omega)=Y(\omega)\quad\text{in }\mathcal{S}'(\R^n),$$
and
\begin{align}\label{cor-random-path-limit}
\lim_{\alpha\to0^+}\alpha\|X(\omega)\|_{\W^{-\alpha,p}}^p
=\frac{2}{p}\|Y(\omega)\|_{\L^p}^p.
\end{align}

It remains to justify that $Y$ can be chosen as an $\L^p(\R^n)$-valued random variable. Let $(t_m)_{m\in\mathbb{N}}$ be a sequence such that $t_m\downarrow0$ as $m\uparrow\infty$. By assumption, for each $m\in\mathbb{N}$, $X\ast\Phi_{t_m}$ has a strongly measurable $\L^p(\R^n)$-valued version. On the full-measure set $\Omega'\in\mathcal{F}$, where \eqref{cor-random-path} holds, Theorem \ref{main-MS} gives $X(\omega)=Y(\omega)\in\L^p(\R^n)$, and Lemma \ref{diff-lem}(1) yields
$$X(\omega)\ast\Phi_{t_m}=Y(\omega)\ast\Phi_{t_m}\rightarrow Y(\omega)\quad\text{in }\L^p(\R^n),$$
as $m\to\infty$. Thus, we may define $Y$ as this limit on $\Omega'$ and set $Y=0$ elsewhere. Since $\L^p(\R^n)$ is separable for $p<\infty$, the pointwise limit of strongly measurable $\L^p(\R^n)$-valued random variables is strongly measurable. Hence $Y$ is an $\L^p(\R^n)$-valued random variable.

Finally, applying Fatou's lemma to \eqref{cor-random-path-limit} gives
\begin{align*}
\E\|Y\|_{\L^p}^p
=\frac{p}{2}\E\left[\lim_{j\to\infty}\alpha_j\|X\|_{\W^{-\alpha_j,p}}^p\right]
\leq\frac{p}{2}\liminf_{j\to\infty}\E\left[\alpha_j\|X\|_{\W^{-\alpha_j,p}}^p\right],
\end{align*}
which proves \eqref{cor-random-est}.
\end{proof}

\begin{remark}\label{remark-cor-random}
We observe that the Sobolev norms in \eqref{cor-random-cond} are measurable. Indeed, by the second assertion of Lemma \ref{diff-lem}(2), for each fixed $\omega\in\Omega$, the map $t\mapsto X(\omega)\ast\Phi_t$ is continuous from $(0,\infty)$ to $\L^p(\R^n)$. By the assumption, for each fixed $t>0$, $X\ast\Phi_t$ admits a strongly measurable $\L^p(\R^n)$-valued version. Since $\L^p(\R^n)$ is separable for $p\in(1,\infty)$, by \cite[Lemma 6.4.6]{Bogachev2007}, we derive that $(\omega,t)\longmapsto X(\omega)\ast\Phi_t$ has a jointly strongly measurable version. Thus, by the definition in \eqref{neg-sob-norm} and the Fubini--Tonelli theorem, $\omega\mapsto \|X(\omega)\|_{\W^{-\alpha_j,p}}^p$ is measurable for every $j\in\mathbb{N}$.
\end{remark}

\begin{example}[Occupation measures of fBm]\label{eg-fbm}
Let $T>0$ and $(\Omega,\mathcal{F},\mathbb{P})$ be a probability space. Let $B^H=(B_t^H)_{t\in[0,T]}$ be a fractional Brownian motion (fBm) on $\R^n$ with Hurst index $H\in(0,1)$, i.e., a centered continuous-time Gaussian process with covariance matrix
$${\rm Cov}(B_t^H,B_s^H)=\frac{1}{2}(s^{2H}+t^{2H}-|t-s|^{2H})\,{\rm I}_n,\quad s,t\geq0,$$
where ${\rm I}_n$ is the $n\times n$ identity matrix. For basics and applications of the fBm, refer to the monograph \cite{BHOZ2008}. For any Borel set $A\subset\R^n$, define the occupation measure of $B^H$ up to time $T$ by
$$\nu_T(A):=\int_0^T\mathbbm{1}_A(B_s^H)\,\d s.$$
Equivalently, for any bounded Borel function $g:\R^n\to\R$,
$$\int_{\R^n}g(x)\,\nu_T(\d x)=\int_0^T g(B_s^H)\,\d s.$$
It records the amount of time that the trajectory spends in each region of the state space $\R^n$. If normalized by $T^{-1}$, $\nu_T/T$ is also known as the empirical measure of
$B^H$. For recent studies of empirical measures of fBm and subordinated fBm under Wasserstein distances on the flat torus, see \cite{HMT2023,LiWu2026}.

A local time up to time $T$ is an occupation density of this measure, namely, a measurable random field $L_T$ such that
$$\nu_T=L_T\,\d x\quad\text{almost surely}.$$
We call it an $\L^2$-local time if, in addition, $L_T\in\L^2(\R^n)$ almost surely. Below, we obtain the stronger integrability
$$L_T\in\L^2(\Omega;\L^2(\R^n)),$$
where $\L^2(\Omega;\L^2(\R^n))$ denotes the Lebesgue--Bochner space consists of strongly measurable functions $F:\Omega\to\L^2(\R^n)$ such that $\E\|F\|_{\L^2(\R^n)}^2<\infty.$
The study of occupation densities and local times is classical; see, for instance,
\cite{Berman1973,GemanHorowitz1980}.

Assume that
$$ Hn<1.$$
Then $\nu_T$ has a square-integrable density. More precisely, there exists a random field
$L_T\in\L^2(\Omega;\L^2(\R^n))$ such that
\begin{align}\label{eg-fbm-1}
\nu_T(\d x)=L_T(x)\,\d x\quad\text{almost surely},
\end{align}
and
\begin{align}\label{eg-fbm-2}
\lim_{\alpha\to0^+}\alpha\|\nu_T\|_{\W^{-\alpha,2}}^2
=\|L_T\|_{\L^2(\R^n)}^2\quad\text{almost surely}.
\end{align}

Indeed, since $\nu_T$ is a finite random Radon measure, it is naturally an
$\mathcal{S}'(\R^n)$-valued random variable. For each $t>0$, by the Fubini--Tonelli theorem and the
semigroup property of the heat kernel $(\Phi_t)_{t>0}$,
\begin{align*}
\|\nu_T\ast\Phi_t\|_{\L^2}^2
&=\int_{\R^n}\int_0^T\int_0^T\Phi_t(x-B_v^H)\Phi_t(x-B_u^H)\,\d u\d v\d x  \\
&=\int_0^T\int_0^T\Phi_{2t}(B_u^H-B_v^H)\,\d u\d v.
\end{align*}
Moreover, for fixed $t>0$, the map
$$\omega\mapsto \nu_T(\omega)\ast\Phi_t=\int_0^T \Phi_t(\,\cdot-B_s^H(\omega))\,\d s$$
is strongly measurable as an $\L^2(\R^n)$-valued random variable.

Since $B_u^H-B_v^H$ is an $\R^n$-valued centered Gaussian vector with covariance matrix $|u-v|^{2H}{\rm I}_n$, we have
$$\E\Phi_{2t}(B_u^H-B_v^H)=\frac{1}{\big[2\pi(4t+|u-v|^{2H})\big]^{n/2}}.$$
Hence, for every $t>0$,
\begin{align*}
\E\|\nu_T\ast\Phi_t\|_{\L^2}^2
&=\frac{1}{(2\pi)^{n/2}}\int_0^T\int_0^T\frac{\d u\,\d v}{(4t+|u-v|^{2H})^{n/2}}  \\
&\leq\frac{1}{(2\pi)^{n/2}}\int_0^T\int_0^T|u-v|^{-Hn}\,\d u\,\d v=:C_{T,n,H}<\infty,
\end{align*}
where the finiteness follows from the assumption $Hn<1$. Consequently, for any $\alpha\in(0,1)$,
\begin{align*}
\E\left[\alpha\|\nu_T\|_{\W^{-\alpha,2}}^2\right]
&=\alpha\int_0^1 t^{\alpha-1}\E\|\nu_T\ast\Phi_t\|_{\L^2}^2\,\d t  \\
&\leq C_{T,n,H}\alpha\int_0^1 t^{\alpha-1}\,\d t=C_{T,n,H}.
\end{align*}
Taking any sequence $\alpha_j\downarrow0$, we obtain
$$\sup_{j\geq1}\E\left[\alpha_j\|\nu_T\|_{\W^{-\alpha_j,2}}^2\right]<\infty.$$
In particular, $\|\nu_T\|_{\W^{-\alpha_j,2}}<\infty$ almost surely for every $j$, and hence, $\nu_T\in\W^{-\alpha_j,2}(\R^n)$ almost surely for every $j$.

Thus, all assumptions of Corollary \ref{cor-random} are satisfied with $p=2$ and $X=\nu_T$. Therefore, there exists an $\L^2(\R^n)$-valued random variable $L_T$ such that
$$\nu_T=L_T\,\d x\quad\text{almost surely},$$
and
$$\E\|L_T\|_{\L^2(\R^n)}^2\leq C_{T,n,H}<\infty.$$
This proves $L_T\in\L^2(\Omega;\L^2(\R^n))$. Finally, the almost sure limiting identity \eqref{eg-fbm-2} follows directly from Theorem \ref{main-MS}.
\end{example}

\subsection{Applications of the abstract formula}
\begin{example}[Truncated Hardy--Littlewood maximal operators]\label{prop-maximal}
Let $p\in(1,\infty)$. For $f\in\L^p(\R^n)$ and $t>0$, define the truncated Hardy--Littlewood maximal operator as
$$M_tf(x):=\sup_{0<r\leq t}\frac{1}{|B(x,r)|}\int_{B(x,r)}|f(y)|\,\d y,\quad x\in\R^n,$$
where $B(x,r)$ denotes the open ball in $\R^n$ with center $x$ and radius $r$, and $|B(x,r)|$ denotes its $n$-dimensional Lebesgue measure. Then
$$\lim_{\alpha\to0^+}\alpha\int_0^1 t^{\frac{\alpha p}{2}-1}\|M_tf\|_{\L^p}^p\,\d t=\frac{2}{p}\|f\|_{\L^p}^p.$$

Indeed, choosing the supremum over rational radii gives an equivalent measurable representative of $M_tf$. The operators $M_t$ are sub-additive and satisfy $M_tf\leq Mf$, where $M$ is the usual Hardy--Littlewood maximal operator, defined by
$$M f(x):=\sup_{r>0}\frac{1}{|B(x,r)|}\int_{B(x,r)}|f(y)|\,\d y,\quad x\in\R^n.$$
 Hence, by the Hardy--Littlewood maximal inequality (see, e.g., \cite[Theorem 2.1.6]{Grafakos2014a}),
$$\|M_tf\|_{\L^p}\leq\|Mf\|_{\L^p}\leq C_{n,p}\|f\|_{\L^p},\quad 0<t\leq1,\,p\in(1,\infty),$$
for some constant $C_{n,p}$ depending only on $n$ and $p$. Moreover, the Lebesgue differentiation theorem (see, e.g., \cite[Corollary 2.1.16]{Grafakos2014a}) gives $M_tf(x)\to |f(x)|$ for almost every $x\in\R^n$ as $t\to0^+$. Since $M_tf\leq Mf$ and $Mf\in\L^p(\R^n)$, we have
$$|M_tf-|f||^p\leq 2^{p-1}\big[(Mf)^p+|f|^p\big],$$
and the right-hand side belongs to $\L^1(\R^n)$. The dominated convergence theorem yields
$$\|M_tf-|f|\|_{\L^p}\to0,\quad\mbox{as }t\to0^+.$$
In particular, $\|M_tf\|_{\L^p}\to\|f\|_{\L^p}$ as $t\to0^+$. Therefore, the limiting formula follows from Theorem \ref{main-gen-MS}.
\end{example}

\begin{example}[An Abel limit for conditional expectations]\label{eg-martingale}
Let $p\in[1,\infty)$, let $(\Omega,\mathcal{F},\mathbb{P})$ be a probability space, and let
$$\mathcal{F}_0\subset\mathcal{F}_1\subset\cdots\subset\mathcal{F}$$
be an increasing sequence of sub-$\sigma$-algebras such that the $\sigma$-algebra generate by $(\mathcal{F}_k)_{k\geq0}$, denoted by  $\mathcal{F}_\infty=\sigma(\cup_{k\geq0}\mathcal{F}_k)$, coincides with $\mathcal{F}$, up to $\mathbb{P}$-null sets. For $f\in \L^p(\Omega,\mathbb{P})$ and $k=0,1,2,\ldots$, set
$$E_kf:=\mathbb{E}[f\mid\mathcal{F}_k].$$
Then
\begin{align}\label{eg-martingale-1}
\lim_{\beta\to0^+}(1-e^{-\beta})\sum_{k=0}^\infty e^{-\beta k}\|E_kf\|_{\L^p(\mathbb{P})}^p
=\|f\|_{\L^p(\mathbb{P})}^p.
\end{align}

We are left to show \eqref{eg-martingale-1}.  Fix $f\in \L^p(\Omega,\mathbb{P})$.  For $0<t\leq1$, let $k(t)$ be the unique integer $k\geq0$ such that
$e^{-(k+1)}<t\leq e^{-k},$ and define
$$S_tf:=E_{k(t)}f,\quad 0<t\leq1.$$
When $t>1$, we set $S_tf:=E_0f$. Then $(S_t)_{t>0}$ is a family of operators on $\L^p(\Omega,\mathbb{P})$.

We now check the hypotheses of Theorem \ref{main-gen-MS}. Choose measurable versions of
$E_kf$. Then
$$S_tf(\omega)=\sum_{k=0}^\infty\mathbbm{1}_{(e^{-(k+1)},e^{-k}]}(t)\,E_kf(\omega),\quad 0<t\leq1,\,\omega\in\Omega,$$
and hence, $(t,\omega)\mapsto S_tf(\omega)$ is jointly measurable. Thus $(\mathfrak{p}.1)$ holds.   Moreover, Jensen's inequality gives
$$\|E_kf\|_{\L^p(\mathbb{P})}\leq \|f\|_{\L^p(\mathbb{P})},
\quad k=0,1,2,\cdots.$$
Hence, Minkowski's inequality implies
$$\|S_tf\|_{L^p(\mathbb{P})}\leq \|f\|_{\L^p(\mathbb{P})},\quad 0<t\leq1,$$
which shows that $(\mathfrak{p}.2)$ holds with constant $C=1$.

By the convergence theorem for conditional expectations with respect to increasing $\sigma$-algebras (see e.g. \cite[Theorem 10.2.1]{Bogachev2007}),
$$E_kf\rightarrow \mathbb{E}[f\mid\mathcal{F}_\infty]\quad\text{in }\L^p(\Omega,\mathbb{P}),$$
as $k\to\infty$. (This can also be obtained by using the martingale convergence theorem, since it is a standard result that $(E_kf, \mathcal{F}_k)_{k\geq0})$ is a martingale; see e.g. \cite[Proposition 3.11 and Corollary 3.18]{URR2025}.) Since $\mathcal{F}_\infty=\mathcal{F}$ up to null sets and $f$ is $\mathcal{F}$-measurable,
we have as $k\to\infty$,
$$E_kf\to f\quad\text{in }\L^p(\Omega,\mathbb{P}).$$
As $t\to0^+$, one has $k(t)\to\infty$, and hence
$$S_tf=E_{k(t)}f\to f\quad\text{in }\L^p(\Omega,\mathbb{P}).$$
In particular, $(\mathfrak{p}.3)$ holds.

We apply Theorem \ref{main-gen-MS}. For $\alpha>0$, set $\beta:=\frac{\alpha p}{2}.$ Then
\begin{align*}
\alpha\|f\|_{\mathcal{W}^{-\alpha,p}(\mathbb{P})}^p
&=\alpha\int_0^1 t^{\frac{\alpha p}{2}-1}\|S_tf\|_{\L^p(\mathbb{P})}^p\,\d t  \\
&=\frac{2\beta}{p}\sum_{k=0}^\infty\|E_kf\|_{\L^p(\mathbb{P})}^p\int_{e^{-(k+1)}}^{e^{-k}} t^{\beta-1}\,\d t  \\
&=\frac{2}{p}(1-e^{-\beta})\sum_{k=0}^\infty \|E_kf\|_{\L^p(\mathbb{P})}^p  e^{-\beta k},
\end{align*}
where the Abel series on right-hand side is finite for every $\beta>0$ in the sense that it is absolutely convergent, since $0\leq \|E_kf\|_{\L^p(\mathbb{P})}^p\leq \|f\|_{\L^p(\mathbb{P})}^p$ for all $k=0,1,\cdots$, and $\sum_{k=0}^{\infty}e^{-\beta k}<\infty$. Thus, Theorem \ref{main-gen-MS} leads to \eqref{eg-martingale-1}.
\end{example}

\subsection*{Acknowledgment}\hskip\parindent
The author thanks Professor Feng-Yu Wang for posing Question (Q) during the author's talk at a workshop in 2025. The main results of this paper were presented at the 15th AIMS Conference on Dynamical Systems, Differential Equations and Applications (Athens, Greece, July 6--10, 2026).  The author also gratefully acknowledges financial support from the National Key R\&D Program of China (Grant No.~2022YFA1006000) and the National Natural Science Foundation of China (Grant No.~12671176).


\begin{thebibliography}{a23}

\bibitem{Berman1973}
S.M. Berman: Local nondeterminism and local times of Gaussian processes. Indiana Univ. Math. J. 23 (1973), no. 1, 69--94.

\bibitem{ACPS20a}
A. Alberico, A. Cianchi, L. Pick, L. Slav\'{\i}kov\'{a}:  On the Limit as $s\to0^+$ of Fractional Orlicz--Sobolev Spaces. J. Fourier Anal. Appl. 26 (2020), Paper No. 80.

\bibitem{ACPS20b}
A. Alberico, A. Cianchi, L. Pick, L. Slav\'{\i}kov\'{a}: On the limit as $s\to1^-$ of possibly non-separable fractional Orlicz--Sobolev spaces. Atti Accad. Naz. Lincei Cl. Sci. Fis. Mat. Natur. 31 (2020), no. 4, 879--899.

\bibitem{AKM2019}
S. Armstrong, T. Kuusi, J.-C. Mourrat: \emph{Quantitative stochastic homogenization and large-scale regularity}. Grundlehren der mathematischen Wissenschaften [Fundamental Principles of Mathematical Sciences], 352. Springer, Cham, 2019.

\bibitem{BGL2014}
D. Bakry, I. Gentil, M. Ledoux: \emph{Analysis and geometry of Markov diffusion operators}. Grundlehren der mathematischen Wissenschaften [Fundamental Principles of Mathematical Sciences], 348. Springer, Cham, 2014.

\bibitem{BHOZ2008}
F. Biagini, Y. Hu, B. {\O}ksendal, T. Zhang: \emph{Stochastic Calculus for Fractional Brownian Motion and Applications}. Probability and its Applications (New York). Springer-Verlag London, Ltd., London, 2008.

\bibitem{Bogachev2007}
V.I. Bogachev: \emph{Measure Theory}, Vol. II. Springer, Berlin, 2007.

\bibitem{BBM1}
J. Bourgain, H. Brezis, P. Mironescu: Another look at Sobolev spaces. In \emph{Optimal control and partial differential equations}, IOS, Amsterdam 2001, 439--455.

\bibitem{BBM2}
J. Bourgain, H. Brezis, P. Mironescu: Limiting embedding theorems for $W^{s,p}$ when $s\nearrow1$ and applications. Dedicated to the memory of Thomas H. Wolff. J. Anal. Math.
87 (2002), 77--101.

\bibitem{BSY23}
D. Brazke, A. Schikorra, P.-L. Yung: Bourgain--Brezis--Mironescu convergence via Triebel-Lizorkin spaces. Calc. Var. Partial Differential Equations 62 (2023), Paper No. 41.

\bibitem{Brezis2002}
H. Brezis: How to recognize constant functions. Connections with Sobolev spaces. Volume in honor of M. Vishik, Uspekhi Mat. Nauk 57 (2002), 59--74; English translation in Russian Math. Surveys 57 (2002), 693--708.

\bibitem{Brezis2011}
H. Brezis: \emph{Functional Analysis, Sobolev Spaces and Partial Differential Equations}. Universitext, Springer, 2011.


\bibitem{BGT2022}
F. Buseghin, N. Garofalo, G. Tralli: On the limiting behaviour of some nonlocal seminorms: a new phenomenon. Ann. Sc. Norm. Super. Pisa Cl. Sci. (5) Vol. XXIII (2022), 837--875.


\bibitem{DGPYYZ24}
F. Dai, L. Grafakos, Z. Pan, D. Yang, W. Yuan, Y. Zhang: The Bourgain--Brezis--Mironescu formula on ball Banach function spaces. Math. Ann. 388 (2024), no. 2, 1691--1768.

\bibitem{Davila02}
J. D\'{a}vila: On an open question about functions of bounded variation.  Calc. Var. Partial Differential Equations 15 (2002), 519--527.


\bibitem{DNPV2012}
E. Di Nezza, G. Palatucci, E. Valdinoci: Hitchhiker's guide to the fractional Sobolev spaces. Bull. Sci. Math. 136 (2012), no. 5, 521--573.

\bibitem{DLTYY2024}
O. Dom\'{\i}nguez, Y. Li, S. Tikhonov, D. Yang, W. Yuan: A unified approach to self-improving property via $K$-functionals.  Calc. Var. Partial Differential Equations 63 (2024), no. 9, Paper No. 231.


\bibitem{Feller2}
W. Feller: \emph{An introduction to probability theory and its applications. II.} Second edition. John Wiley \& Sons, Inc., New York, London, Sydney, 1971.

\bibitem{FL2023}
F. Flandoli, E. Luongo: \emph{Stochastic partial differential equations in fluid mechanics}. Lecture Notes in Mathematics, 2330. Springer, Singapore, 2023.

\bibitem{FOT2011}
M. Fukushima, Y. Oshima, M. Takeda: \emph{Dirichlet forms and symmetric Markov processes}. De Gruyter Studies in Mathematics, 19. Walter de Gruyter \& Co., Berlin, extended edition, 2011.

\bibitem{GemanHorowitz1980}
D. Geman, J. Horowitz: Occupation densities. Ann. Probab. 8 (1980), no. 1, 1--67.

\bibitem{GT2024}
N. Garofalo, G. Tralli: A universal heat semigroup characterisation of Sobolev and BV spaces in Carnot groups. Int. Math. Res. Not. IMRN no. 8, (2024), 6731--6758.

\bibitem{Grafakos2014a}
L. Grafakos: \emph{Classical Fourier analysis}. Third edition. Graduate Texts in Mathematics, 249. Springer, New York, 2014.

\bibitem{Grafakos2014b}
L. Grafakos:  \emph{Modern Fourier Analysis}.  Third edition. Graduate Texts in Mathematics, 250. Springer, New York, 2014.

\bibitem{Han24}
B.-X. Han: On the asymptotic behaviour of the fractional Sobolev seminorms: A geometric approach. J. Funct. Analysis
287 (2024), no. 9, Paper No. 110608.

\bibitem{HPXZ25}
B.-X. Han, A. Pinamonti, Z. Xu, K. Zambanini: Maz'ya--Shaposhnikova meet Bishop--Gromov. Potential Anal. 63 (2025), 513--529.

\bibitem{HLYY25}
P. Hu, Y. Li, D. Yang, W. Yuan: A sharp localized weighted inequality related to Gagliardo and Sobolev seminorms and its applications.
Adv. Math. 481 (2025), Paper No. 110537.

\bibitem{HMT2023}
M. Huesmann, F. Mattesini, D. Trevisan: Wasserstein asymptotics for the empirical measures of fractional Brownian motion on a flat torus.  Stochastic Process. Appl. 155 (2023), 1--26.


\bibitem{HvNVW2016}
T. Hyt\"{o}nen, J. van Neerven, M. Veraar, L. Weis:\emph{ Analysis in Banach spaces. Volume I: Martingales and Littlewood-Paley theory.}  A Series of Modern Surveys in Mathematics,
63. Springer, Cham, 2016.

\bibitem{Leoni2023}
G. Leoni:  \emph{A first course in fractional Sobolev spaces}. Graduate Studies in Mathematics, 229. American Mathematical Society, Providence, RI, 2023.


\bibitem{LiWu2025+}
H. Li, B. Wu: On Limit Formulas for Besov Seminorms and Nonlocal Perimeters in the Dunkl Setting. Preprint (2025), arXiv:2503.20809v2.

\bibitem{LiWu2026}
H. Li, B. Wu: Wasserstein convergence for empirical measures of subordinated fractional Brownian motions on the flat torus. Bull. Sci. math. 213 (2026), Paper No. 103891, 44 pp.

\bibitem{Ludwig2014}
M. Ludwig: Anisotropic fractional Sobolev norms. Adv. Math. 252 (2014), 150--157.

\bibitem{MR1992}
Z.-M. Ma, M. R\"{o}ckner: \emph{Introduction to the Theory of (Non-Symmetric) Dirichlet Forms}. Springer, 1992.

\bibitem{MS2002}
V. Maz'ya, T. Shaposhnikova: On the Bourgain, Brezis, and Mironescu theorem concerning limiting embeddings of fractional Sobolev spaces. J. Funct. Anal. 195 (2002), 230--238.

\bibitem{Milman05}
M. Milman: Notes on limits of Sobolev spaces and the continuity of interpolation scales. Trans. Amer. Math. Soc. 357 (2005), 3425--3442.

\bibitem{Mitrea2018}
D. Mitrea: \emph{Distributions, partial differential equations, and harmonic analysis}. Second edition. Universitext, Springer, New York, 2018.


\bibitem{Mohanta24}
K. Mohanta: Bourgain-Brezis-Mironescu formula for $W_q^{s,p}$-spaces in arbitrary domains.  Calc. Var. Partial Differential Equations 63 (2024), no. 2, Paper No. 31.

\bibitem{NTYYZ25+}
E. Nakai, M. Tang, D. Yang, W. Yuan, C. Zhu: Maz'ya--Shaposhnikova Representation of Quasi-Norms of Ball Quasi-Banach Function Spaces on Spaces of Homogeneous Type with Weak Reverse Doubling Property. Preprint (2025), arXiv:2511.22960.

\bibitem{Oleinik2025}
R. Oleinik: Asymptotic relations of the Bourgain-Brezis-Mironescu type for mappings between singular spaces. J. Geom. Anal. 35 (2025), no. 196.

\bibitem{Temam2001}
R. Temam: \emph{Navier-Stokes equations: theory and numerical analysis}. AMS Chelsea Publishing, 343. Providence, RI, 2001 (Reprint of the 1984 edition).

\bibitem{Tenenbaum15}
 G. Tenenbaum: \emph{Introduction to analytic and probabilistic number theory}. Third edition. Graduate Studies in Mathematics, 163. American Mathematical Society, Providence, RI, 2015.

\bibitem{Triebel92}
H. Triebel: \emph{Theory of function spaces. II.}  Monographs in Mathematics, 84. Birkh\"{a}user-Verlag, Basel, 1992.


\bibitem{URR2025}
W. Urbina-Romero, R. Rios: \emph{An Introduction to the Modern Martingale Theory and Applications: An Analytic View.} Texts in Applied Mathematics, 81. Springer International Publishing, 2025.

\bibitem{Yoshida}
K. Yoshida: \emph{Functional Analysis}. Sixth Edition. Springer-Verlag, Berlin, 1980.
\end{thebibliography}
\end{document}